\documentclass[a4paper, reqno]{amsart}

\usepackage{amsmath,amsthm,amssymb,xspace,stmaryrd}
\usepackage{tikz}
\usetikzlibrary{positioning}

\tikzstyle{mor} = [execute at begin node=$\scriptstyle, execute at end node=$]

\usepackage[pdftitle=Duality\ for\ convex\ monoids,
  pdfauthor={Frank Roumen and Sutanu Roy},
  pdfsubject={Mathematics; MSC }
]{hyperref}
\usepackage[lite]{amsrefs}
\newcommand*{\MRref}[2]{ \href{http://www.ams.org/mathscinet-getitem?mr=#1}{MR #1}}
\newcommand*{\arxiv}[1]{ \href{http://www.arxiv.org/abs/#1}{arXiv:#1}}
\renewcommand{\PrintDOI}[1]{\href{http://dx.doi.org/#1}{DOI #1}
  \IfEmptyBibField{volume}{, (to appear in print)}{}}

\newcommand{\blank}{-}
\newcommand{\dd}{\,\mathrm{d}}
\newcommand{\id}{\mathrm{id}}

\newcommand{\inprod}[2]{\langle #1 \mid #2 \rangle}

\newcommand{\Cat}[1]{\ensuremath{\mathbf{#1}\xspace}}

\newcommand{\FinSets}{\Cat{FinSets}}

\newcommand{\KHaus}{\Cat{KHaus}}
\newcommand{\Cstar}{\Cat{C^*}}
\newcommand{\CCstar}{\Cat{cC^*}}
\newcommand{\KGrp}{\Cat{KGrp}}
\newcommand{\KQGrp}{\Cat{KQGrp}}
\newcommand{\CKQGrp}{\Cat{CKQGrp}}

\newcommand{\KConv}{\Cat{KConv}}

\newcommand{\EMod}{\Cat{EMod}}
\newcommand{\BEMod}{\Cat{BEMod}}

\newcommand{\op}{^{\mathrm{op}}} 

\newcommand{\Dst}{\mathcal{D}}
\newcommand{\Rad}{\mathcal{R}}

\newcommand{\B}{\mathcal{B}}

\newcommand{\Ef}{\mathcal{E}\!f}

\newcommand{\DM}{\mathcal{D\!M}}

\newcommand{\RR}{\mathbb{R}}
\newcommand{\CC}{\mathbb{C}}

\newcommand{\basis}[1]{\lambda_{#1}}

\DeclareMathOperator{\Hom}{Hom}

\DeclareMathOperator{\End}{End}
\DeclareMathOperator{\Aut}{Aut}

\DeclareMathOperator{\tr}{tr}
\DeclareMathOperator{\Spec}{Spec}
\DeclareMathOperator{\St}{St}

\newtheorem{theorem}{Theorem}
\newtheorem{lemma}[theorem]{Lemma}

\newtheorem{proposition}[theorem]{Proposition}
\theoremstyle{definition}
\newtheorem{definition}[theorem]{Definition}
\newtheorem{example}[theorem]{Example}
\newtheorem{examples}[theorem]{Examples}
\theoremstyle{remark}

\newcommand*{\Cst}{\mathrm C^*}
\newcommand*{\nb}{\nobreakdash}

\usepackage{xcolor}

\title{Duality for convex monoids}
\begin{document}
\author{Frank Roumen}
\email{f.roumen@math.ru.nl}
\address{Inst. for Mathematics, Astrophysics, and Particle Physics (IMAPP)\\
 Radboud University Nijmegen\\
 Heyendaalseweg 135, 6525 AJ Nijmegen\\
 The Netherlands}    
\author{Sutanu Roy}
\email{sr26@uottawa.ca}
\address{School of Mathematics and Statistics\\
 Carleton University \\
 1125 Colonel By Drive, Ottawa, ON K1S 5B6\\
 Canada}
 
 \begin{abstract}
     Every C*-algebra gives rise to an effect module and a convex space of
     states, which are connected via Kadison duality. We explore this duality
     in several examples, where the C*-algebra is equipped with the structure
     of a finite-dimensional Hopf algebra. When the Hopf algebra is the function
     algebra or group algebra of a finite group, the resulting state spaces
     form convex monoids. We will prove that both these convex monoids can be
     obtained from the other one by taking a coproduct of density matrices on
     the irreducible representations. We will also show that the same holds
     for a tensor product of a group and a function algebra. 
 \end{abstract}
 

\maketitle

\section{Introduction}
 \label{sec:intro}
States and observables of a physical system are connected via dualities
between certain categories. There are several dualities that can 
be used for this connection. 
Known examples include the Gelfand duality theorem and the Kadison duality
theorem. For a system in classical physics, 
the state space is modeled by a topological space, and the
observables are given by functions on this space. In this way, the algebra of
observables forms a commutative \(\Cst\)\nb-algebra. 
The celebrated Gelfand theorem states that the category of locally compact Hausdorff 
spaces is dually equivalent to the category of commutative 
\(\Cst\)\nb-algebras, and thus it provides an intimate connection between
states and observables. A useful special case occurs when the
\(\Cst\)\nb-algebras under consideration have a unit. Gelfand duality in this
setting states that the category of unital \(\Cst\)\nb-algebras is dually
equivalent to the category of compact Hausdorff spaces.

Gelfand duality does not apply to quantum mechanical systems, since their
algebra of observables is in general a non-commutative \(\Cst\)\nb-algebra. 
There is no good non-commutative analogue of Gelfand duality, but there is a duality theorem
due to Kadison that can be useful to describe quantum systems. Kadison duality is
not based on \(\Cst\)\nb-algebras, but on the unit interval within a unital
\(\Cst\)\nb-algebra. This
unit interval forms a structure called an \emph{effect module}, and there is a dual
equivalence between a certain category  of 
effect modules and a certain category  of 
\emph{convex spaces}. The state space of a quantum system forms a convex space and 
the corresponding effect module contains its observables; hence
Kadison duality connects states 
and observables of quantum systems. It does not directly generalize Gelfand duality, 
since the unit interval of a \(\Cst\)\nb-algebra contains less information than the 
\(\Cst\)\nb-algebra itself.

When studying physical systems, one often wants to take the symmetry group of
the system into account. In the \(\Cst\)\nb-algebraic picture, this leads to quantum
groups. For ordinary Gelfand duality, we use locally compact Hausdorff spaces 
as state spaces.
If we take the symmetry of a system into account, the state space
becomes a locally compact group. On the dual side, this gives a coalgebra
structure on the \(\Cst\)\nb-algebra, making it into a structure called a quantum
group. There is an analogue of the Gelfand duality theorem that takes the
symmetry into account. This theorem states that the category of compact
(Hausdorff) groups is dually equivalent to the category of commutative compact
quantum groups. 

Summarizing, there are two dualities involving topological spaces and
\(\Cst\)\nb-algebras: one for systems without symmetry, and one for systems with
symmetry. Furthermore, Kadison duality relates convex spaces and effect
modules for systems without symmetry. In this article we shall will describe a 
variant of Kadison duality for systems with symmetry. This will lead to a notion of a quantum
group whose underlying algebra is an effect module instead of a \(\Cst\)\nb-algebra.
Schematically, we wish to complete the following diagram:
\begin{center}
    \begin{tikzpicture}
        \node (KHaus) {$\KHaus$};
        \node (CCstar) [right=of KHaus] {$\CCstar\op$};
        \node (KConv)  [below=of KHaus] {$\KConv$};
        \node (BEMod)  [below=of CCstar] {$\BEMod\op$};

        \node (KGrp) [right=of CCstar] {$\KGrp$};
        \node (CKQGrp) [right=of KGrp] {$\CKQGrp\op$};
        \node (question1) [below=of KGrp] {?};
        \node (question2) [below=of CKQGrp] {?$\op$};

        \path [->] (KHaus.north east) edge  node [mor, above] {C}
                   (CCstar.north west)
                   (CCstar.south west) edge  node [mor, below] {\Spec}
                   (KHaus.south east)
                   (KConv.north east) edge node [mor, above] {\Hom(-,[0,1])}
                   (BEMod.north west)
                   (BEMod.south west) edge node [mor, below] {\Hom(-,[0,1])}
                   (KConv.south east)
                   (KHaus) edge node [mor, left] {\Rad} (KConv)
                   (CCstar) edge node [mor, right] {[0,1]_{(-)}} (BEMod)

                   (KGrp.north east) edge  node [mor, above] {C}
                   (CKQGrp.north west)
                   (CKQGrp.south west) edge  node [mor, below] {\Spec}
                   (KGrp.south east)
                   (question1.north east) edge (question2.north west)
                   (question2.south west) edge (question1.south east)
                   (KGrp) edge node [mor, left] {\Rad} (question1)
                   (CKQGrp) edge node [mor, right] {[0,1]_{(-)}} (question2);
        \draw [white] (KHaus) -- node [black] {$\simeq$} (CCstar);
        \draw [white] (KConv) -- node [black] {$\simeq$} (BEMod);
        \draw [white] (KGrp) -- node [black] {$\simeq$} (CKQGrp);
        \draw [white] (question1) -- node [black] {$\simeq$} (question2);
    \end{tikzpicture}
\end{center}
The categories and functors occuring in this diagram will be explained in more
detail in the next section.
We will restrict our attention to finite groups.  In the theory of \(\Cst\)\nb-algebraic 
quantum groups, there is only one way to assign a commutative quantum group or Hopf\nb-algebra 
to any finite group. We show that there are two ways to assign an effect module 
(and a dual convex space) to a finite group, arising from two different
Hopf algebras associated to the group. Both ways to form ``effect quantum
groups'' are related via a version of Pontryagin duality. 

The outline of this paper is as follows. Section~\ref{sec:Preliminaries}
contains preliminary material about convex spaces, effect modules, and quantum
groups. In particular we will describe the various dualities that connect
these objects. In Section~\ref{sec:KadisonGroupAlg} we will determine the
effect modules and convex spaces associated to the group algebra and the
function algebra of a finite group. The two convex spaces obtained in this way
are both convex monoids, that is, monoids in the category of convex spaces.
The connection between these two monoids will be established in
Section~\ref{sec:ConvexPontryaginDuality}. We will prove 
that both convex monoids determine each other via essentially the
same construction: if $V_1, \ldots, V_k$ are the irreducible linear
representations of either of these monoids, then the coproduct $\DM(V_1) +
\cdots + \DM(V_k)$ is a convex monoid isomorphic to the other one. Finally, in
Section~\ref{sec:TensorProductDuality}, we will prove a related result for the
tensor product of a group and a function algebra.

\section{Preliminaries}
\label{sec:Preliminaries}

We will present the dualities alluded to in the Introduction in more detail
here. The most basic duality that we will use is Gelfand duality. Throughout
this paper, we will assume that all \(\Cst\)\nb-algebras we encounter have a unit.
Write $\Cstar$ for the category of \(\Cst\)\nb-algebras with *-homomorphisms as maps.
The full subcategory of commutative \(\Cst\)\nb-algebras is denoted $\CCstar$.
Furthermore write $\KHaus$ for the category of compact Hausdorff spaces with
continuous maps. If $X$ is a compact Hausdorff space, then the collection
$C(X)$ of complex-valued functions on $X$ is a commutative \(\Cst\)\nb-algebra with
pointwise operations. This construction gives a contravariant functor $C$ from
$\KHaus$ to $\CCstar$ by letting it act on morphisms via precomposition. The
Gelfand spectrum provides a functor in the other direction: if $A$ is a
commutative \(\Cst\)\nb-algebra, then its spectrum $\Spec A = \Hom_{\CCstar}(A,\CC)$ is
a compact Hausdorff space. The spectrum construction forms a contravariant
functor from $\CCstar$ to $\KHaus$, again using precomposition.

\begin{theorem}[Gelfand]
    \label{thm:GelfandDuality}
    The compositions $C \circ \Spec$ and $\Spec \circ C$ are naturally
    equivalent to the identity functor. Hence the categories $\KHaus$ and
    $\CCstar$ are dually equivalent.
\end{theorem}

There is a more general version of Gelfand duality involving non-unital
\(\Cst\)\nb-algebras and locally compact spaces, but we will only be concerned
with compact spaces in the remainder of this article.

The Gelfand Theorem justifies viewing \(\Cst\)\nb-algebras as a non-commutative
generalization of spaces. Similarly it is useful to have a non-commutative
generalization of topological groups. This gives the notion of a quantum
group. There are several definitions of quantum groups; here we will use the
compact quantum groups from Woronowicz \cite{Woronowicz:Compact_pseudogroups}. For a general
overview of the theory of quantum groups see \cite{Timmermann:An_inv_to_QG_and_duality}.

\begin{definition}
    A \emph{compact quantum group} is a \(\Cst\)\nb-algebra $A$ equipped with a
    *-homomorphism $\Delta : A \to A \otimes A$ called the
    \emph{comultiplication}, such that
    \begin{itemize}
        \item The comultiplication is coassociative, i.e.\ $(\Delta \otimes
            \id_A) \circ \Delta = (\id_A \otimes \Delta) \circ \Delta$.
        \item The linear spans of $\Delta(A) (1 \otimes A)$ and $\Delta(A) (A
            \otimes 1)$ are dense in $A \otimes A$.
    \end{itemize}
\end{definition}

If $G$ is a compact Hausdorff group, then its function algebra $C(G)$ is a
commutative \(\Cst\)\nb-algebra. It can be made into a compact quantum group by
defining $\Delta : C(G) \to C(G) \otimes C(G) \cong C(G \times G)$,
$\Delta(\varphi)(g,h) = \varphi(gh)$. This construction provides a
group-theoretic analogue of Gelfand duality. Instead of compact 
spaces, we use compact groups. They constitute a category $\KGrp$
with continuous homomorphisms as maps. Morphisms between compact quantum
groups are unital $*$-homomorphisms preserving the comultiplication. They make
compact quantum groups into a category $\KQGrp$. As in Gelfand duality, we
want to consider the full subcategory $\CKQGrp$ of commutative compact quantum
groups.

\begin{theorem}
    \label{thm:QuantumGroupDuality}
    The functor $C : \KGrp\op \to \CKQGrp$ is a dual equivalence between the
    category of compact Hausdorff groups and commutative compact quantum
    groups.
\end{theorem}

If $A$ is a commutative compact quantum group, then the underlying space of
its dual group is the spectrum of $A$, considered as \(\Cst\)\nb-algebra. The
multiplication on $G$ arises from the comultiplication on $A$.

There is another way to assign a compact quantum group to a finite group $G$,
namely the group algebra $C[G]$. The elements are again functions from $G$ to
$\CC$, but now the multiplication is given by convolution:
\[ \varphi * \psi(g) = \sum_{hk = g} \varphi(h) \psi(k) \]
The standard basis of $C[G]$ consists of Dirac functions $\basis{g}$ for $g\in
G$, defined by $\basis{g}(g) = 1$ and $\basis{g}(h) = 0$ for $h \neq g$. The
convolution product assumes a particularly easy form on these basis vectors,
namely $\basis{g} * \basis{h} = \basis{gh}$.  The comultiplication is defined
on basis vectors by $\Delta(\basis{g}) = \basis{g} \otimes \basis{g}$.

\subsection{Effect algebras and modules}
 \label{subsec:eff_mod}
Another duality that we will use involves the effects in a \(\Cst\)\nb-algebra. Effects
represent probabilistic measurements that can be performed on a physical
system. Let $A$ be any \(\Cst\)\nb-algebra. An element $a$ in $A$ is said to be
\emph{positive} if it can be written as $a = b^*b$ for some $b\in A$.
Positivity can be used to define an order on the self-adjoint part of $A$,
called the \emph{L\"owner order}. Let $a,b$ be self-adjoint elements in $A$,
then we say that $a \leq b$ if and only if $b-a$ is positive. An \emph{effect}
in $A$ is a self-adjoint $a \in A$ for which $0 \leq a \leq 1$.

Effects in a \(\Cst\)\nb-algebra can be organized into an algebraic structure called an
\emph{effect module}. Effect modules were introduced in \cite{Gudder-Pulmannova:Rep_the_conv}, 
based on earlier work on effect algebras, which started in \cite{Foulis-Bennett:Effect_alg_qnt_log}. 
For an overview of the theory about effect algebras, see \cite{Dvurecenskij-Pulmannov:New_trends}.

Roughly speaking, an effect module looks like a vector
space, but the addition is only a partial operation (since the sum of two
effects may lie above 1), and we can only multiply by scalars in the unit
interval $[0,1]$. Instead of complements with respect to $0$, we have
complements with respect to $1$. This means that for every effect $a$ there
exists an effect $b$ for which $a + b = 1$. The precise definition is as
follows.

\begin{definition}
    An \emph{effect module} consists of a set $A$ equipped with a partial
    binary operation $\boxplus$ called addition, a unary operation
    $(\blank)^{\bot}$ called orthocomplement, a scalar multiplication $\cdot :
    [0,1] \times A \to A$ and constants $0,1 \in A$, subject to the following
    axioms:
    \begin{itemize}
        \item The operation $\boxplus$ is commutative, which means that
            whenever $a \boxplus b$ is defined, then also $b \boxplus a$ is
            defined, and $a \boxplus b = b\boxplus a$.
        \item The operation $\boxplus$ is associative, which means that if $a
            \boxplus b$ and $(a \boxplus b) \boxplus c$ are defined, then also
            $b \boxplus c$ and $a \boxplus (b \boxplus c)$ are defined, and
            $(a \boxplus b) \boxplus c = a \boxplus (b\boxplus c)$.
        \item For every $a \in A$, $a \boxplus 0 = 0 \boxplus a = a$.
        \item For all $a,b \in A$, $a \boxplus b = 1$ if and only if $b =
            a^\bot$.
        \item If $a\boxplus 1$ is defined, then $a=0$.
        \item For all $r,s \in [0,1]$ and $a\in A$, $r \cdot (s \cdot a) =
            (rs) \cdot a$.
        \item If $r+s \leq 1$, then $(r + s)\cdot a = r\cdot a + s\cdot a$.
        \item If $a \boxplus b$ is defined, then $r \cdot (a \boxplus b) = r
            \cdot a \boxplus r \cdot b$.
        \item $1 \cdot a = a$.
    \end{itemize}
    Effect modules form a category $\EMod$, in which the morphisms are
    functions preserving addition, orthocomplement, scalar multiplication, and
    the constants $0$ and $1$.
\end{definition}

The easiest example of an effect module is the unit interval $[0,1]$. The
partial operation is addition, where $a \boxplus b$ is defined if and only if
$a + b \leq 1$. The orthocomplement is given by $a^\bot = 1-a$, and the scalar
multiplication is simply the multiplication on $[0,1]$. Another example are
the effects in a \(\Cst\)\nb-algebra, with the same operations. If $A$ is a \(\Cst\)\nb-algebra,
then its collection of effects is denoted $\Ef(A)$. Any Hilbert space $H$
gives rise to a \(\Cst\)\nb-algebra $\B(H)$, hence to an effect module $\Ef(\B(H))$. We
will often abbreviate this to $\Ef(H)$.

More generally, every partially ordered vector space $V$ over $\RR$ gives rise
to an effect module.  Pick an element $u\in V$ for which $u > 0$, then the
interval $[0,u] = \{ v \in V \mid 0 \leq v \leq u \}$ is an effect module.
Addition serves as the partial binary operation, and the orthocomplement is
$v^\bot = u-v$. The scalar multiplication is obtained by restricting the
scalar multiplication from $\RR$ to $[0,1]$. In fact, every effect module is
an interval in some partially ordered $\RR$-vector space, as shown in
\cite[Theorem 3.1]{Gudder-Pulmannova:Rep_the_conv}.

To work with infinite-dimensional vector spaces, it is often necessary to
require that they are complete in a certain metric. The same holds for effect
modules. If $A$ is an effect module, then a \emph{state} on $A$ is a morphism
$\sigma : A \to [0,1]$. The collection of all states is written as $\St(A)$.
Define a metric on $A$ via 
\[ d(a,b) = \sup_{\sigma \in \St(A)} | \sigma(a) - \sigma(b) | \]
We call the effect module $A$ a \emph{Banach effect module} if it is complete in
its associated metric. Banach effect modules give a full subcategory of
$\EMod$ written as $\BEMod$.

\subsection{Convex spaces}
\label{subsec:convex_spaces}
The state space of an effect module is always a compact convex space. We will
make this observation more precise by defining a suitable category of compact
convex spaces, following \cite{Semadeni:Categorical_convexity}. A topological
vector space is said to be \emph{locally convex} if its topology has a base of
convex open sets. Let $\KConv$ be the category whose objects are compact
convex subspaces of a locally convex vector space. A subspace $X \subset V$ is
called convex if, for all $x,y \in X$ and $\lambda \in [0,1]$, we have that
$\lambda x + (1-\lambda) y \in X$. A morphism between compact convex spaces $X
\subset V$ and $Y \subset W$ is a continuous map $f : X \to Y$ that preserves
convex combinations, i.e. $f(\lambda x + (1-\lambda)y) = \lambda f(x) +
(1-\lambda) f(y)$. Such a map is called \emph{affine}.

The state space of an effect module $A$ is contained in the vector space $\{
\varphi : A \to \RR \mid \varphi(a \boxplus b) = \varphi(a) + \varphi(b) \}$,
which is locally convex. Therefore $\St(A)$ is an object in the category
$\KConv$, and $\St$ is a contravariant functor from $\BEMod$ to $\KConv$. The
functor $\Hom_{\KConv}(\blank, [0,1])$ is a contravariant functor in the other
direction. The following result is taken from
\cite[Theorem 6]{Jacobs-Mandemaker:Expec_monad}, but see also
\cite[Section 4]{Semadeni:Categorical_convexity}.

\begin{theorem}
    The functors $\St$ and $\Hom_{\KConv}(\blank, [0,1])$ are inverses of each
    other. Hence the categories $\KConv$ and $\BEMod$ are dually equivalent.
    \label{thm:KadisonDuality}
\end{theorem}

\begin{examples}
    \label{ex:KadisonDuality}
    We give some examples of convex spaces and their dual effect modules.
    \begin{enumerate}
        \item If $X$ is a finite set, then let $\Dst(X) = \{ f : X \to [0,1]
            \mid \sum_{x\in X} f(x) = 1 \}$. This can be visualized as the
            standard simplex whose vertices are points in $X$. An element $f
            \in \Dst(X)$ is usually written as a formal convex combination
            $\sum_{x\in X} a_x x$, where the coefficients are the function
            values $a_x = f(x)$. They are subject to the condition $\sum_x a_x
            = 1$. This construction gives a functor $\Dst : \FinSets \to
            \KConv$, where on a morphism $\varphi : X \to Y$ we define
            $\Dst(\varphi)(\sum_x a_x x) = \sum_x a_x \varphi(x)$.
            The dual effect module of $\Dst(X)$ is $\Hom(\Dst(X), [0,1])$,
            which is isomorphic to $[0,1]^X$.
        \item In the above example, $\Dst(X)$ can be thought of as the set of
            discrete probability measures or distributions on $X$. There is a
            continuous analogue of this construction. Let $X$ now be a compact
            Hausdorff space, and let $\Sigma_X$ be its Borel $\sigma$-algebra.
            Denote the space of Radon measures on $X$ by $\Rad(X)$. A Radon
            measure is a probability measure $\mu : \Sigma_X \to [0,1]$ that
            satisfies
            \[ \mu(M) = \sup_{\substack{K \subseteq M \\ K \text{ compact}}} 
                \mu(K). \]
            In \cite{Furber-Jacobs:Prob_Gelfand_duality} it is shown that
            $\Rad$ forms a monad on the category of compact Hausdorff spaces.
            Its category of Eilenberg-Moore algebras is equivalent to
            $\KConv$, so convex spaces of the form $\Rad(X)$ can be thought of
            as the free convex spaces over a compact Hausdorff space. The dual
            effect module of $\Rad(X)$ is the collection of continuous
            functions from $X$ into $[0,1]$. This fact is a categorical
            reformulation of the Riesz-Markov theorem. To see this, observe
            that there is a map $\Rad(X) \to \Hom(C(X,[0,1]), [0,1])$ given by
            integration, i.e. $\mu \mapsto \int (-) \dd \mu$. The Riesz-Markov
            theorem states that this map is an isomorphism, so $\Rad(X)$ is
            the dual of $C(X,[0,1])$.
            This shows that the following
            diagram, connecting Gelfand and Kadison duality, commutes:
\begin{center}
    \begin{tikzpicture}
        \node (KHaus) {$\KHaus$};
        \node (CCstar) [right=of KHaus] {$\CCstar\op$};
        \node (KConv)  [below=of KHaus] {$\KConv$};
        \node (BEMod)  [below=of CCstar] {$\BEMod\op$};

        \path [->] (KHaus.north east) edge  node [mor, above] {C}
                   (CCstar.north west)
                   (CCstar.south west) edge  node [mor, below] {\Spec}
                   (KHaus.south east)
                   (KConv.north east) edge node [mor, above] {\Hom(-,[0,1])}
                   (BEMod.north west)
                   (BEMod.south west) edge node [mor, below] {\Hom(-,[0,1])}
                   (KConv.south east)
                   (KHaus) edge node [mor, left] {\Rad} (KConv)
                   (CCstar) edge node [mor, right] {[0,1]_{(-)}} (BEMod);

        \draw [white] (KHaus) -- node [black] {$\simeq$} (CCstar);
        \draw [white] (KConv) -- node [black] {$\simeq$} (BEMod);
    \end{tikzpicture}
\end{center}
        \item Let $H$ be a Hilbert space. A density matrix on $H$ is a
            positive trace-class operator $\rho : H \to H$ with trace 1. The
            collection of all density matrices forms a convex space denoted
            $\DM(H)$. The importance of this example lies in its connection to
            the effects on $H$: there is an isomorphism $\Ef(H) \to
            \Hom(\DM(H), [0,1])$, that maps an effect $a$ to the function
            $\rho \mapsto \tr(\rho a)$. Because this map is an isomorphism,
            $\Ef(H)$ is the dual effect module of $\DM(H)$.
    \end{enumerate}
\end{examples}

There are several ways to construct new convex spaces from old ones. In the
remainder of this paper we will sometimes use coproducts and tensor products
of convex spaces, so we will describe these briefly here.

The category $\KConv$ has all coproducts. The coproduct of two convex spaces
can be described geometrically, using the embedding in a locally convex vector
space. The following description is a slight modification of the
construction in \cite{Semadeni:Categorical_convexity}. Suppose that $X
\subseteq V$ and $Y \subseteq W$ are compact convex subsets of locally convex
vector spaces. Then the coproduct $X + Y$ can be embedded in the vector space
$V \oplus W \oplus \mathbb{R}$. To construct this coproduct, embed $X$ in this
larger vector space via the inclusion $x \mapsto (x,0,1)$, and embed $Y$ via
the inclusion $y \mapsto (0,y,0)$. The convex hull of the disjoint union of
$X$ and $Y$ is the coproduct of $X$ and $Y$. This is made precise in the
following.

\begin{proposition}
    If $X \subseteq V$ and $Y \subseteq W$ are objects in the category
    $\KConv$, then their coproduct is
    \[ X + Y = \{ (rx, (1-r)y, r) \mid r\in [0,1],\ x\in X,\ y\in Y \} \subseteq
        V \oplus W \oplus \mathbb{R}. \]
\end{proposition}
\begin{proof}
    Define embeddings $i_X : X \to X+Y$ and $i_Y : Y \to X + Y$ via $i_X(x) =
    (x,0,1)$ and $i_Y(y) = (0,y,0)$. Given affine maps $f : X \to Z$ and $g :
    Y \to Z$, define $h : X + Y \to Z$ by
    \[ h(rx, (1-r)y, r) = r f(x) + (1-r) g(y). \]
    Then $h \circ i_X = f$ and $h \circ i_Y = g$, so it remains to be shown
    that $h$ is the unique map with this property. Suppose that $h' : X + Y
    \to Z$ is an affine map for which $h \circ i_X = f$ and $h \circ i_Y = g$.
    Then 
    \[ h'(rx, (1-r)y, r) = h'\left( r (x,0,1) + (1-r) (0,y,0) \right) = r f(x)
    + (1-r) g(y), \]
    which proves uniqueness.
\end{proof}

\begin{example}
    Denote the one-point convex space by $1$. The coproduct $1 + \cdots + 1$
    of $n$ copies of this space is the convex hull of $n$ points, embedded in
    $\RR^{n-1}$ in such a way that they are all affinely independent.
    Therefore this coproduct is the standard simplex $\Dst(n)$.
\end{example}

We continue with a discussion of the tensor product of compact convex spaces.
If $X$, $Y$, and $Z$ are compact convex spaces, then a map $X \times Y \to Z$
is called \emph{bi-affine} is it is affine in both variables separately. A
\emph{tensor product} of $X$ and $Y$ is a compact convex space $X \otimes Y$
equipped with a bi-affine map $\otimes : X \times Y \to X \otimes Y$ such that
for every compact convex space $Z$ and every bi-affine $f : X \times Y \to Z$
there exists a unique affine map $g : X \otimes Y \to Z$ such that $g \circ
\otimes = f$. Semadeni proves in \cite{Semadeni:Categorical_convexity} that
any two compact convex spaces admit a tensor product, and that it is unique up
to isomorphism.

The above tensor product enjoys many good properties. The one-point convex
space $1$ acts as a unit for the tensor. Furthermore, the tensor product
distributes over coproducts. From these two facts, together with the
isomorphism $\Dst(n) \cong 1 + \cdots + 1$, it can be deduced that the tensor
product of standard simplices is $\Dst(n) \otimes \Dst(m) \cong \Dst(nm)$.

\section{Kadison duality for group and function algebras}
\label{sec:KadisonGroupAlg}

Let~\(G\) be a finite group. This gives rise to two Hopf\nb-algebras, or
compact quantum groups, namely the function algebra $C(G)$ and the group
algebra $C[G]$.
Of these two Hopf\nb-algebras, the function algebra is commutative but in
general not cocommutative, while for the group algebra, it is the other way
round. Therefore the duality from Theorem~\ref{thm:QuantumGroupDuality} only
applies to the function algebra $C(G)$. However, Kadison duality also applies to unit 
intervals of non-commutative \(\Cst\)\nb-algebras, so we can use this for both the group 
algebra and the function algebra. 

\begin{definition}
    A \emph{convex monoid} is an object $X$ of the category $\KConv$, together
    with a continuous multiplication map $\cdot : X \times X \to X$ and a
    constant $1\in X$, such that
    \begin{itemize}
        \item The operation $\cdot$ is affine in both variables separately,
            that is, $(\lambda x + (1-\lambda)y) \cdot z = \lambda x\cdot z +
            (1-\lambda) y\cdot z$ and similarly for convex combinations on the
            right.
        \item The operation $\cdot$ is associative.
        \item $1$ is a unit for $\cdot$.
    \end{itemize}
\end{definition}

Equivalently, a convex monoid is a convex space $X$ equipped with a map $X
\otimes X \to X$ that is associative and has a unit.

A variant of quantum groups in the framework of Kadison duality should 
give a duality between effect modules with a comultiplication and convex
monoids. In this section we describe these objects 
for the Hopf\nb-algebras \(C(G)\) and~\(C[G]\). We will start with the
function algebra $C(G)$. In fact, this algebra can be defined for any compact
group $G$, so we will now determine the effect module and convex space
associated to $C(G)$ for an arbitrary compact group $G$.
\begin{proposition}
 \label{the:eff_mod_st_sp}
 Let~\(G\) be a compact group. Then 
 \begin{enumerate}
  \item the effect module 
  \(\Ef(C(G))\cong\{\varphi\mid\text{ \(\varphi\colon G \to [0,1]\) is continuous} \}\).  
  Restriction of~\(\Delta\) on~\(\Ef(C(G))\) defines a comultiplication  
  map~\(\Ef(C(G))\to \Ef(C(G\times G))\) (which is a morphism of effect modules); 
  \item the state space~\(\St(\Ef(C(G))\) is isomorphic to the 
  space $\Rad(G)$ of Radon measures on $G$. 
  Moreover, $\Rad(G)$ is a convex monoid with respect to the 
  multiplication obtained by dualizing~\(\Delta\) on~\(\Ef(C(G))\).
 \end{enumerate}
\end{proposition}
\begin{proof}
 The effect module $\Ef(C(G))$ consists of all
 functions $\varphi$ for which $0 \leq \varphi \leq 1$. Since the
multiplication in $C(G)$ is pointwise, the order is also pointwise, and hence
$\Ef(C(G))$ consists of continuous maps $G \to [0,1]$. The comultiplication on
$C(G)$ induces a map of effect modules $ \Delta: \Ef(C(G)) \to \Ef(C(G \times
G))$, given by $\Delta(\varphi)(g,h) = \varphi(gh)$. Clearly, \(\Delta\) is
coassociative.

The state space of $\Ef(C(G))$ consists of all morphisms $\sigma :
\Hom(G,[0,1]) \to [0,1]$. By part 2 of Examples~\ref{ex:KadisonDuality}, this
is the same as the space of Radon measures $\Rad(G)$.


\end{proof} 

The multiplication on $\Rad(G)$ can also be decribed directly in terms of the
multiplication on $G$. Applying the functor $\Rad$ to the multiplication map
$\cdot : G \times G \to G$ gives a map $\Rad(G \times G) \to \Rad(G)$. Since
$\Rad(G) \otimes \Rad(G) \cong \Rad(G \times G)$, this provides a convex
monoid structure on $\Rad(G)$, which is the dual of $\Ef(C(G))$. 
This convex monoid has been studied categorically in \cite{Mislove:Prob_monads}.

The multiplication on the group algebra $C[G]$ is more complicated than the
one on the function algebra. Therefore the L\"owner order on $C[G]$
and the effect module are also more difficult to compute explicitly.
The algebra $C[G]$ is simultaneously a C*-algebra and a Hilbert space, and the
algebra structure is compatible with the inner product, so $C[G]$ forms a
Hilbert algebra.  We shall use some general facts about Hilbert algebras to
compute the effect module and the state space of~\(C[G]\). 

\begin{definition}
    A \emph{Hilbert algebra} is a *-algebra $A$ equipped with an inner product
    $\inprod{-}{-}$, such that
    \begin{enumerate}
        \item For all $a,b \in A$, $\inprod{a}{b} = \inprod{b^*}{a^*}$.
        \item For each $a \in A$, the map $b \mapsto ab$ is a bounded
            operator.
        \item For all $a,b,c \in A$, $\inprod{ab}{c} = \inprod{b}{a^*c}$.
        \item The linear span of $\{ ab \mid a,b \in A \}$ is dense in $A$.
    \end{enumerate}
\end{definition}
For more about Hilbert algebras see~\cite{Takesaki_Op_Alg_2}.
We will mainly work with \emph{unital} Hilbert algebras, in which the fourth
property holds automatically.

\begin{lemma}
    Let $A$ be a Hilbert algebra, and $f: A \to A$ a map of left $A$-actions,
    i.e. a map satisfying $f(ab) = a f(b)$. Then:
    \begin{enumerate}
        \item If $f$ is positive, then $\sqrt{f}$ is also a map of left
            $A$-actions.
        \item The adjoint $f^\dag$ is a map of left $A$-actions.
        \item $f^\dag(1) = f(1)^*$.
    \end{enumerate}
    \label{LemHilbertAlgProperties}
\end{lemma}
\begin{proof}
    \mbox{}
    \begin{enumerate}
        \item The square root $\sqrt{f}$ commutes with every operator that
            commutes with $f$.
        \item It suffices to prove that $f^\dag(ab)$ and $af^\dag(b)$ have the
            same inner product with any $x\in A$. This holds because
            \begin{align*}
                \inprod{f^\dag(ab)}{x} = \inprod{ab}{f(x)} = \inprod{b}{a^*
                    f(x)} = \inprod{b}{f(a^*x)} &= \inprod{f^\dag(b)}{a^*x}\\
                     &= \inprod{af^\dag(b)}{x}.
            \end{align*}
        \item $\inprod{f(1)^*}{x} = \inprod{x^{*}}{f(1)} = 
                   \inprod{1}{xf(1)}=\inprod{1}{f(x)} = \inprod{f^\dag(1)}{x}$.
            \qedhere
    \end{enumerate}
\end{proof}

Since any Hilbert algebra $A$ is a *-algebra, it can be ordered, and hence we
can speak about effects in the algebra. These are elements $a\in A$ such that
$0 \leq a \leq 1$. But $A$ is also a Hilbert space, so we can also speak about
effects on the Hilbert algebra, which are maps $\varepsilon : A \to A$ that
lie between $0$ and $\id_A$. The next result connects effects in $A$ with
effects on $A$.

\begin{proposition}
 \label{Prop:equiv_eff}
    Let $A$ be a unital Hilbert algebra. There is a bijective 
    correspondence between:
    \begin{enumerate}
        \item Effects in $A$, i.e. $a\in A$ for which $0 \leq a \leq 1$;
        \item Effects $\varepsilon : A \to A$ that are also maps of left
            $A$-actions.
    \end{enumerate}
\end{proposition}
\begin{proof}
    If $\varepsilon : A \to A$ is an effect for which $\varepsilon(ab) = a
    \varepsilon(b)$, then $\varepsilon(1)$ is an effect in $A$. To show this,
    we will start by proving positivity. The effect $\varepsilon$ has a
    positive square root $\sqrt{\varepsilon}$.  We claim that
    $\sqrt{\varepsilon}(1) \sqrt{\varepsilon}(1) = \varepsilon(1)$.
    This follows from the following computation, using
    Lemma~\ref{LemHilbertAlgProperties}:
    \begin{align*}
        \inprod{\varepsilon(1)}{x} 
        &= \inprod{\sqrt{\varepsilon}\sqrt{\varepsilon}(1)}{x} 
        =\inprod{\sqrt{\varepsilon}(1)}{\sqrt{\varepsilon}^\dag (x)}\\ 
        &= \inprod{\sqrt{\varepsilon}(1)}{x \sqrt{\varepsilon}^\dag(1)}
        = \inprod{\sqrt{\varepsilon}(1)}{x \sqrt{\varepsilon}(1)^*} \\
        &=\inprod{\sqrt{\varepsilon}(1)x^*}{\sqrt{\varepsilon}(1)^*}
        =\inprod{x^{*}}{\sqrt{\varepsilon}(1)^*\sqrt{\varepsilon}(1)^*}
        = \inprod{\sqrt{\varepsilon}(1)\sqrt{\varepsilon}(1)}{x}.
    \end{align*}
    This shows that $\varepsilon(1)$ has a square root, so it is positive.
    Similarly, since the square root $\sqrt{I-\varepsilon}$ exists, the
    element $1-\varepsilon(1) \in A$ is positive. Therefore $\varepsilon(1)$
    is an effect.

    Conversely, if $a\in A$ is an effect, define $\varepsilon : A \to A$ by
    $\varepsilon(x) = xa$. Then $\varepsilon$ is clearly a map of left
    actions. Since $a$ is positive, there is a $b$ such that $a = b^* b$.
    Define $\beta : A \to A$ by $\beta(x) = xb$. Then $\beta\beta^\dag(x) = x
    b^* b = \varepsilon(x)$, so $\varepsilon$ is positive. Analogously we can 
    prove $\varepsilon \leq I$, hence $\varepsilon$ is an effect. It is easy
    to see that both constructions are mutually inverse.
\end{proof}

\begin{proposition}
    \label{PropEquivariantEffects}
    Let $V$ be a unitary representation of $G$. Write the decomposition of
    $V$ into irreducible representations as $V = n_1 V_1 \oplus \cdots \oplus
    n_k V_k$. Then the effect module $\{ \varepsilon : V \to V \mid
    \varepsilon \text{ is effect and intertwiner} \}$ is isomorphic to
    $\Ef(\mathbb{C}^{n_1}) \times \cdots \times \Ef(\mathbb{C}^{n_k})$.
\end{proposition}
\begin{proof}
    An intertwining effect $\varepsilon : V \to V$ can be written as a matrix
    of maps $\varepsilon_{ij} : n_i V_i \to n_j V_j$. By Schur's Lemma, each
    $\varepsilon_{ij} = 0 $ for $i \neq j$. The effects $\varepsilon_{ii}$ can
    in turn be decomposed into an $n_i \times n_i$ matrix of maps $V_i \to
    V_i$, and these are all scalar multiples of the identity by Schur's Lemma.
    Therefore each $\varepsilon_{ii}$ corresponds to an effect on
    $\mathbb{C}^{n_i}$.
\end{proof}
\begin{theorem}
    \label{CorStateSpaceGroupAlgebra}
    Let $G$ be a finite group, and let $V_1, \ldots, V_k$ be its irreducible
    representations. Then 
    \begin{enumerate} 
     \item the effect module $\Ef(C[G])$ is isomorphic to 
    $\Ef(V_1) \times \cdots \times \Ef(V_k)$. The comultiplication map
    $\hat{\Delta}\colon\Ef(C[G]) \to \Ef(C[G \times G])$ given by $\hat{\Delta} 
    \left( \sum_g a_g \basis{g} \right) = \sum_g a_g \basis{(g,g)}$;
    \item The state space of $C[G]$, denoted by~\(\St(C[G])\), 
    is the coproduct $\DM(V_1) +\cdots + \DM(V_k)$ in the category of convex 
    spaces. Moreover, \(\St(C[G])\) is a convex monoid with respect to the 
    multiplication~$\mu\colon\St(C[G \times G]) \to \St(C[G])$ defined as 
    the linear extension of $\mu(\sigma)(\basis{g}) = \sigma(\basis{(g,g)})$.
  \end{enumerate}  
\end{theorem}
\begin{proof}
    By Proposition~\ref{Prop:equiv_eff}, the effect module $\Ef(C[G])$ is isomorphic
    to $\{ \varepsilon : C[G] \to C[G] \mid \varepsilon \text{ is effect and
    map of left $C[G]$-actions} \}$. The condition that $\varepsilon$ is a map
    of left $C[G]$-actions means that it is an intertwiner from the regular
    $G$-representation $C[G]$ to itself. The regular representation decomposes 
    as $C[G] = n_1 V_1 \oplus \cdots \oplus n_k V_k$, so by
    Proposition~\ref{PropEquivariantEffects} $\Ef(C[G])$ is isomorphic to
    $\Ef(\mathbb{C}^{\dim V_1}) \times \cdots \times \Ef(\mathbb{C}^{\dim
    V_k}) \cong \Ef(V_1) \times \cdots \times \Ef(V_k)$. This is a Banach
    effect module, so we can use the duality between convex compact spaces and
    Banach effect modules to determine the dual space. Dualizing turns
    products into coproducts, so the dual space is $\DM(V_1) + \cdots +
    \DM(V_k)$. Clearly, \(\hat{\Delta}\) is coassociative. 
    
    Dualizing~\(\hat{\Delta}\) gives~\(\mu\) on the state space 
    \(\St(C[G])\).  The convex monoid structure on the state space of the
    group algebra satisfies $(\sigma \cdot \tau)(\basis{g}) =
    \sigma(\basis{g})\tau(\basis{g})$ on basis vectors. 
\end{proof}
\section{Convex Pontryagin duality for group and function algebras}
\label{sec:ConvexPontryaginDuality}

The group algebra and the function algebra associated to a finite group are
both finite-dimensional Hopf algebras. These are related via a non-commutative
generalization of Pontryagin duality, see e.g. \cite{Timmermann:An_inv_to_QG_and_duality} for
details. In the previous section, we found two convex monoids that can be
obtained from a finite group: the state space $\Dst(G)$ of the function
algebra, and the state space $\DM(V_1) + \cdots + \DM(V_k)$ of the group
algebra, where the $V_i$ are the irreducible representations of $G$. This
section will present a construction to convert these two convex monoids into
each other. This construction can be viewed as a convex counterpart of
Pontryagin duality.

\begin{definition}
    A \emph{linear representation} of a convex monoid $X$ consists of a vector
    space $V$ and a monoid homomorphism $\rho : X \to \Aut(V)$ that preserves
    convex combinations.
\end{definition}

As usual, a representation can also be written as an action of $X$ on $V$,
that is, a map $X \times V \to V$. A linear representation of a convex monoid
is then required to be affine in the first variable and linear in the second
variable. We will look at the linear representations of the convex monoid
$\Dst(G)$.

\begin{lemma}
    There is a one-to-one correspondence between representations of the finite
    group $G$ and linear representations of $\Dst(G)$.
    \label{lem:LinearRepresentationsDstG}
\end{lemma}
\begin{proof}
    Representations of $G$ are monoid homomorphisms $G \to \Aut(V)$, since all
    monoid homomorphisms between groups are automatically group homomorphisms.
    Linear representations of $\Dst(G)$ are monoid homomorphisms $\Dst(G) \to
    \Aut(V)$ that are also morphisms of convex spaces. Since $\Dst(G)$ is the
    free convex space generated by $G$, it follows that there is a one-to-one
    correspondence between maps of sets $G \to \Aut(V)$ and maps of convex
    spaces $\Dst(G) \to \Aut(V)$. It is easy to check that this equivalence
    restricts to monoid homomorphisms.
\end{proof}

This result produces an easy way to construct the state space of $C[G]$ out of
the state space of $C(G)$, in the following steps:
\begin{enumerate}
    \item Let $V_1$, $\ldots$, $V_k$ be the irreducible linear representations
        of $\St(C(G))$.
    \item Form the convex sets of density matrices $\DM(V_i)$ for each $i$.
    \item The coproduct (in the category $\KConv$) of all $\DM(V_i)$ is the
        state space of $C[G]$.
\end{enumerate}
Since irreducible representations of $G$ are the same as irreducible linear
representations of $\Dst(G)$, this construction yields exactly the state space
of $C[G]$. A surprising fact is that it works in two directions: if we apply
exactly the same construction to the state space of $C[G]$, we end up with the
state space of $C(G)$.

\begin{proposition}
    Let $V_1, \ldots, V_k$ be the irreducible linear representations of the
    convex monoid $\St(C[G])$. Then the convex space $\DM(V_1) + \cdots +
    \DM(V_k)$ is isomorphic to $\St(C(G))$.
    \label{prop:StateSpaceFunctionAlgebraFromGroupAlgebra}
\end{proposition}
\begin{proof}
    We will determine the irreducible representations $V_i$.
    Recall that the multiplication on $\St(C[G])$ was given by $\sigma \cdot
    \tau(\basis{g}) = \sigma(\basis{g}) \tau(\basis{g})$. Therefore this
    convex monoid is commutative. All irreducible representations of a
    commutative monoid are 1-dimensional.
    Each $g \in G$ gives a 1-dimensional linear representation $\rho_g :
    \St(C[G]) \to \CC$ by $\rho_g(\sigma) = \sigma(g)$.

    We will now check that all 1-dimensional linear representations are of the
    form $\rho_g$ for some $g\in G$. Let $\rho : \St(C[G]) \to \mathbb{C}$ be
    an arbitrary representation. Then the map $\rho$ extends to a function
    $\Hom(C[G], \mathbb{C}) \to \mathbb{C}$ in the double dual of $C[G]$,
    hence there exists $a \in C[G]$ such that $\sigma(a) = \rho(\sigma)$ for
    all states $\sigma$ on $C[G]$. We will show that $a$ is actually an
    element in $G \subseteq C[G]$. Express $a$ as $a = a_1 \lambda_{g_1} +
    \cdots + a_n \lambda_{g_n}$. Then, for any two states $\sigma$ and $\tau$,
    \[ \rho(\sigma \tau) = \sigma \tau(\sum_i a_i \lambda_{g_i}) = \sum_i a_i
    \sigma(\lambda_{g_i}) \tau(\lambda_{g_i}) \]
    and
    \[ \rho(\sigma) \rho(\tau) = \sigma(\sum_i a_i \lambda_{g_i}) \tau(\sum_j
        a_j \lambda_{g_j}) = \sum_{i,j} a_i a_j \sigma(\lambda_{g_i})
        \tau(\lambda_{g_j}). \]
    The map $\rho$ is a representation, so these two expressions must be equal
    for all states $\sigma$ and $\tau$. Comparing coefficients shows that at
    most one $a_i$ is equal to 1, and all others are 0. The element $a$ cannot
    be identically 0, since $\rho$ preserves 1. Hence $a$ is equal to
    $\lambda_g$ for some $g\in G$, which proves that the maps $\rho_g$ are
    indeed the only 1-dimensional representations.

    There is only one density matrix on any 1-dimensional space. Therefore the
    space $\DM(V_1) + \cdots + \DM(V_k)$ is a coproduct of $\# G$ copies of
    the one-point space, which is $\Dst(G)$.
\end{proof}

We have shown that if we start with the convex monoid $\St(C(G)) \cong
\Dst(G)$ and apply the above construction twice, then we get back a convex
space that is isomorphic to the underlying space of the original convex
monoid. Now we wish to show that the multiplication is also preserved in this
construction, so that we obtain an isomorphism of convex monoids, rather than
just convex spaces. For this we have to endow the coproduct of density
matrices with a multiplication. It is useful to have an explicit isomorphism
between $\DM(V_1) + \cdots + \DM(V_k)$ and $\St(C[G])$. 

\begin{lemma}
    Let $(V_1, \rho_1), \ldots, (V_k,\rho_k)$ be the irreducible
    representations of $G$.  The map $\Phi : \DM(V_1) + \cdots + \DM(V_k) \to
    \St(C[G])$ determined by $\Phi(T)(\lambda_g) = \tr(T \rho_i(g))$
    for $T \in \DM(V_i)$ is an isomorphism of convex spaces.
\end{lemma}
\begin{proof}
    Consider the map $\Psi: C[G] \to \End(V_1) \times \cdots \times
    \End(V_k)$ of \(\Cst\)\nb-algebras, on basis vectors determined by $\lambda_g
    \mapsto (\rho_1(g), \ldots, \rho_k(g))$. We claim that this map is
    injective. Suppose that $a,b \in C[G]$ are such that $\rho_i(a) =
    \rho_i(b)$ for all $i$.  Then $a$ and $b$ act in the same way in all
    irreducible representations of $G$. Since any representation of $G$ can be
    decomposed into irreducibles, $a$ and $b$ act in the same way in all
    representations of $G$. In particular, they have the same action on the
    regular representation $C[G]$. Thus $a = a \cdot e = b \cdot e =
    b$. Since $\Psi$ is injective and its domain has the same dimension as its
    codomain, it is an isomorphism.

    Taking states of a \(\Cst\)\nb-algebra provides a contravariant functor $\St :
    \Cstar \to \KConv$. Therefore, applying the state functor to $\Psi$ gives
    a map $\St(\End(V_1)) + \cdots + \St(\End(V_k)) \to \St(C[G])$.
    There is an isomorphism $\alpha : \DM(V_i) \to \St(\End(V_i))$ given by
    $\alpha(\rho)(A) = \tr(\rho A)$, and hence $\St(\Psi) = \Phi$. Since
    $\Psi$ is an isomorphism and $\St$ is a functor, $\Phi$ is also an
    isomorphism.
\end{proof}

Using this isomorphism, the multiplication on the coproduct of density
matrices can be described explicitly. Since we are working in a coproduct, it
suffices to describe $T \cdot S$, where $T \in \DM(V_i)$ and $S \in \DM(V_j)$.
Applying the isomorphism $\Phi$ from the lemma above gives states $\lambda_g
\mapsto \tr(T \rho_i(g))$ and $\lambda_g \mapsto \tr(S \rho_j(g))$ on
$\CC[G]$. Multiplying these states pointwise and using properties of the trace
gives the map $\lambda_g \mapsto \tr((T \otimes S) (\rho_i \otimes
\rho_j)(g))$. Since $\Phi$ is an isomorphism, there is a unique $\sum_i
\lambda_i T_i \in \DM(V_1) + \cdots + \DM(V_k)$ for which
\[ \sum_i \lambda_i \tr(T_i \rho_i(g)) = \tr((T \otimes S) (\rho_i \otimes
    \rho_j)(g)). \]
We define $T \cdot S$ to be this convex combination $\sum_i \lambda_i T_i$.
With the proposition and lemma above, we have now proven the following result.

\begin{theorem}
    Let $G$ be a finite group,
    and let $V_1, \ldots, V_k$ be the irreducible linear representations of
    the convex monoid $\St(C(G))$. Then the convex monoid $\DM(V_1) + \cdots +
    \DM(V_k)$ with multiplication described above is isomorphic to the convex
    monoid $\St(C[G])$ with pointwise multiplication.
\end{theorem}

\section{Convex Pontryagin duality for a tensor product}
\label{sec:TensorProductDuality}

The category of finite-dimensional Hopf algebras is self-dual. The dual of a
finite-dimensional Hopf algebra $A$ is $\hat{A} = \{ f : A \to \CC \mid f
\text{ linear} \}$. Its multiplication is derived from the comultiplication on
$A$, and vice versa.
The Hopf algebras $C(G)$ and $C[G]$ coming from a finite group $G$ are duals
of each other via this construction.

Let $A$ be either $C(G)$ or $C[G]$. The main result from the previous section
states that if $V_1, \ldots, V_k$ are the irreducible representations of the
convex monoid $\St(A)$, then $\DM(V_1) + \cdots + \DM(V_k)$ is isomorphic to
$\St(\hat{A})$. This raises the question if this holds for all Hopf algebras.
We do not yet know if this is the case in general, but we will now discuss
another example of a Hopf algebra for which it holds, so this may be promising
for the general case.

Let $G$ be a finite group. Consider the Hopf algebra $A = C(G) \otimes C[G]$,
i.e.  the tensor product of the function algebra and the group algebra. This
Hopf algebra is neither commutative nor cocommutative. Since dualizing
preserves tensor products, the dual of $A$ is isomorphic to $A$ itself. Thus
the statement that connects the state space of $A$ to its dual amounts to the
following.

\begin{proposition}
    Let $V_1, \ldots, V_k$ be the irreducible representations of the convex
    monoid $\St(C(G) \otimes C[G])$. Then $\DM(V_1) + \cdots + \DM(V_k)$ is
    isomorphic to $\St(C(G) \otimes C[G])$.
\end{proposition}
\begin{proof}
    We will first show that $\St(C(G) \otimes C[G])$ is isomorphic to
    $\St(C(G)) \otimes \St(C[G])$. Since the C*-algebra $C(G)$ is
    commutative and finite-dimensional, it is isomorphic to $\CC^n$ for some
    $n$. Hence we have
    \[ \St(C(G) \otimes C[G]) \cong \St(\CC^n \otimes C[G]) \cong
        \St(C[G]^{\oplus n}). \]
    The state space of a direct sum is the coproduct of state spaces, so this
    is isomorphic to
        \begin{align*}
            \St(C[G]) + \cdots + \St(C[G]) &\cong
            (1 + \cdots + 1) \otimes \St(C[G]) \\
            &\cong \Dst(n) \otimes \St(C[G]) \\
            &\cong \St(C(G)) \otimes \St(C[G]).
        \end{align*}
    
    The irreducible representations of a tensor product are tensor products of
    irreducible representations. The irreducible representations of
    $\St(C(G))$ are precisely those of $G$; call these $W_1, \ldots, W_m$.
    There are $\# G = n$ irreducible representations of $C[G]$ and these are
    all one-dimensional. Denote these by $W'_1, \ldots, W'_n$. Then the
    irreducible representations of $\St(C(G) \otimes C[G])$ are the tensor
    products $W_i \otimes W'_j$. Therefore the sum of density matrices is
        \begin{align*}
            \sum_{i,j} \DM(W_i \otimes W'_j) &\cong \left( \sum_i \DM(W_i)
            \right)^{+ n} \\
            &\cong  (1 + \cdots + 1) \otimes \sum_i \DM(W_i) \\
            & \cong  \Dst(G) \otimes \sum_i \DM(W_i) \\
            & \cong  \St(C(G)) \otimes \St(C[G]) \\
            & \cong  \St(C(G) \otimes C[G]),
        \end{align*}
    which is what we wanted to show.
\end{proof}

\section*{Acknowledgements}
We would like to thank Johan Commelin, Robert Furber and Philip Scott for
helpful discussions.  The first author has been financially supported by the
Netherlands Organisation for Scientific Research (NWO) under TOP-GO grant 
no.\ 613.001.013 (The logic of composite quantum systems). The second author 
has been supported by Fields--Ontario postdoctoral fellowship, NSERC and ERA at 
the University of Ottawa and Carleton University.

\end{document}